\documentclass[pdflatex, sn-mathphys-num]{sn-jnl}

\usepackage{graphicx}%
\usepackage{multirow}%
\usepackage{amsmath,amssymb,amsfonts}%
\usepackage{amsthm}%
\usepackage{mathrsfs}%
\usepackage[title]{appendix}%
\usepackage{xcolor}%
\usepackage{textcomp}%
\usepackage{manyfoot}%
\usepackage{booktabs}%
\usepackage{algorithm}%
\usepackage{algorithmicx}%
\usepackage{algpseudocode}%
\usepackage{listings}%
\usepackage{bigints}
\newcommand{\norm}[1]{\left\lVert #1\right\rVert}

\theoremstyle{thmstyleone}%
\newtheorem{theorem}{Theorem}[section]
\newtheorem{lemma}[theorem]{Lemma}

\newtheorem{proposition}[theorem]{Proposition}

\theoremstyle{thmstyletwo}%

\theoremstyle{thmstylethree}%

\begin{document}

\title[Approximations of eigenvalue problem]{Projection-based approximations for eigenvalue problems of Fredholm integral operators with Green's kernels}

\author[1]{\fnm{Shashank K.} \sur{Shukla}}\email{shashankks@rgipt.ac.in}

\author*[1]{\fnm{Gobinda} \sur{Rakshit}}\email{g.rakshit@rgipt.ac.in}

\author[2]{\fnm{Akshay S.} \sur{Rane}}\email{as.rane@ictmumbai.edu.in}

\affil[1]{\orgdiv{Department of Mathematical Sciences}, \orgname{Rajiv Gandhi Institute of Petroleum Technology}, \orgaddress{\street{Jais}, \city{Amethi} \postcode{229304}, \state{Uttar Pradesh}, \country{India}}}

\affil[2]{\orgdiv{Department of Mathematics}, \orgname{Institute of Chemical Technology}, \orgaddress{\street{Nathalal Parekh Marg, Matunga}, \city{Mumbai} \postcode{400019}, \state{Maharashtra}, \country{India}}}


\abstract{We consider the eigenvalue problem $\mathcal{K} x = \lambda x$. Our analysis focuses on the convergence rates of eigenvalue and spectral subspace approximations for compact linear integral operator $\mathcal{K}$ with Green's kernels. By employing orthogonal and interpolatory projections at \( 2r+1 \) collocation points (which are not necessarily Gauss points) onto an approximating space of piecewise even degree polynomials, we establish the superconvergence of eigenfunctions under iteration. The modified projection methods achieve a faster convergence rates compared to classical projection methods. The enhancement in convergence rate is verified by numerical examples.}




\maketitle

\setcounter{equation}{0}
\section{Introduction}\label{intro}

Let $\mathcal{X} = L^2[0,1] ~\text{or}~ C[0, 1]$, and $\Omega = [0, 1] \times [0, 1]$.  
Let $\kappa$ be a real-valued continuous function on $\Omega$.  
We aim to solve the eigenvalue problem of finding an eigenvector $\varphi \in \mathcal{X}$ and an eigenvalue $\lambda \in \mathbb{C} \setminus \{0\}$ such that  
\begin{equation} \label{Eq: 001}
	\mathcal{K} \varphi = \lambda \varphi , 
\end{equation}
where $\mathcal{K}: \mathcal{X} \mapsto \mathcal{X}$ is a compact linear integral operator, defined by  
\begin{equation} \label{Eq: 002}
	(\mathcal{K}x)(s) = \int_{0}^{1} \kappa(s, t)x(t) \, dt, \quad s \in [0, 1], \quad x \in \mathcal{X},
\end{equation}
with $\kappa$ being a specified kernel function of Green's type.  

Let $\mathcal{K}_n$ be a sequence of continuous finite-rank operators converging pointwise to $\mathcal{K}$, and suppose the family $\{\mathcal{K}_n\}$ is collectively compact, meaning the set  
\[
\{\mathcal{K}_n x : n \geq 1, \ \|x\|_\infty \leq 1\}
\]
has compact closure in $\mathcal{X}$. The eigenvalue problem \eqref{Eq: 001} is approximated by  
\begin{equation} \label{Eq: 003}
	\mathcal{K}_n \varphi_n = \lambda_n \varphi_n, \quad \lambda_n \in \mathbb{C}, \quad 0 \neq \varphi_n \in \mathcal{X}.
\end{equation}

Since $\mathcal{K}_n$ is of finite rank, the eigenvalue problem \eqref{Eq: 003} can be reduced to a matrix eigenvalue problem (see Atkinson \cite{KEA0} and Ahues et al. \cite{Ahues}). The approximate solutions $(\lambda_n, \varphi_n)$ converge to the exact solution $(\lambda, \varphi)$ under standard spectral approximation theory for compact operators.  

When establishing the convergence of an approximate solution, a key consideration is often the rate at which the approximate solution converges to the exact solution.
Several important methods have been developed for computing approximate solutions, including Nystr\"om methods, where the integral in \eqref{Eq: 002} is replaced by a convergent quadrature formula, and projection-based approaches such as Galerkin and collocation methods, along with their variants. The analysis of approximate solutions for eigenvalue problems with smooth kernels has been widely studied in the literature \cite{Ahues, KEA0, Baker0, Chatelin, Sloan0}. For kernels of the Green's function-type, which often lack differentiability along the diagonal, early investigations were conducted in Chatelin--Lebbar \cite{Cha-Leb2}, where interpolatory projections used Gauss points (zeros of Legendre polynomials) as collocation nodes.  

The improvement of spectral subspace approximations in projection methods through iterative techniques was first introduced by Sloan \cite{Sloan1}. In Kulkarni \cite{RPK5, RPK6}, a modified projection method was proposed for solving \eqref{Eq: 001} using sequences of projections. It was shown that for sufficiently smooth kernels, the modified projection method yields faster convergence of eigenelements compared to the iterated projection method. Furthermore, convergence of eigenvectors can be significantly enhanced using the iterated modified projection method, obtained from a one-step iteration of the modified projection approach. Other notable approaches include the degenerate kernel method in \cite{Gnan1}, which employs kernel interpolation in both variables and provides error bounds for eigenvalues and spectral subspaces, while also demonstrating improved convergence through eigenfunction iteration. Nystr\"om and degenerate kernel methods using projections at Gauss points onto spaces of piecewise polynomials are known to achieve high-order convergence. Moreover, in~\cite{CA-PS-DS-MT}, superconvergent variants of these methods were introduced, further enhancing the convergence of eigenvalues and spectral subspaces for sufficiently smooth kernels, with additional gains in eigenvector convergence through iteration. In all these cases, the use of piecewise polynomial basis functions reduces the problem to finite-dimensional matrix eigenvalue computations.

Consider a uniform partition of $[0,1]$ given by $\{0 = t_0 < t_1 < \cdots < t_n = 1\}$, where $t_j = \dfrac{j}{n}$ for $j = 0, 1, 2, \ldots, n$, and let $h = \dfrac{1}{n}$. For $r \geq 0$, let $\mathcal{X}_n$ denote the space of piecewise polynomials of even degree $\leq 2r$ associated with this partition.  

In this paper, we establish convergence rates of eigenvalue and spectral subspace approximations for the eigenvalue problem \(\mathcal{K} \varphi = \lambda \varphi\) using both orthogonal projections and interpolatory projections at \(2r+1\) collocation points. In contrast to earlier works such as~\cite{RPK5,RPK6}, which mainly consider smooth kernels together with piecewise discontinuous approximating spaces, we focus on integral operators with Green's function-type kernels. Moreover, instead of employing Gauss points, we adopt a simpler setup with \(2r+1\) equidistant points in each subinterval \([t_{j-1}, t_j]\), \(j=1,2,\ldots,n\). Despite this choice, our analysis show that iterations of both classical collocation/Galerkin and modified projection approximations of the eigenfunctions still yield superconvergence. Furthermore, the modified projection method achieves faster convergence than classical projection methods for eigenvalue approximations. These findings highlight the flexibility and effectiveness of our approach, demonstrating that high-order convergence can be attained without relying on special-function-based collocation nodes.

The structure of the paper is as follows.  Section \(2\) introduces the definition of Green's function-type kernels and spectral projections, along with key results and fixed notations used throughout the paper. Section \(3\) defines both orthogonal and interpolatory projections and provides relevant theoretical insights. In the final subsection of this section, we present projection methods and their variants tailored for eigenvalue problems. Section \(4\) details our main results on the orders of convergence for the Galerkin and collocation methods, including their modified versions. Numerical illustration is provided in Section \(5\). The paper concludes with Section \(6\), summarizing the findings and implications of our study.


\section{Preliminaries}

Let $\mathcal{X} = L^2[0,1] ~\text{or}~ C[0, 1]$. Consider the subsets  of $\Omega$ as
\[
\Omega_1 = \{(s,t) \mid 0 \leq t \leq s \leq 1\},  
\quad  
\Omega_2 = \{(s,t) \mid 0 \leq s \leq t \leq 1\}.
\]  
Suppose there exist functions \(\kappa_i \in C^{\alpha}(\Omega_i)\), \(i=1,2\), such that the kernel \(\kappa\) can be written as  
\[
\kappa(s,t) =  
\begin{cases}  
	\kappa_1(s,t), & (s,t) \in \Omega_1, \\[1mm]
	\kappa_2(s,t), & (s,t) \in \Omega_2,  
\end{cases}  
\]  
and \(
\kappa_1(s,s) = \kappa_2(s,s), ~ \text{for all } s \in [0,1].
\)
For \( r \geq 0 \), we assume that the smoothness parameter \(\alpha\) satisfies \(\alpha > r\). A kernel \(\kappa\) satisfying these properties is referred to as a Green's kernel. Such kernels are generally non-smooth, as their regularity is restricted to the separate domains \(\Omega_1\) and \(\Omega_2\), with only continuity across the diagonal $s=t$.

Now, we will consider some basic definitions and stadard results on spectral theory:\\
Let $\mathcal{X} = L^2[0,1] ~\text{or}~ C[0, 1]$, and $\mathcal{BL}(\mathcal{X})$ will denote the set of all bounded linear operators on $\mathcal{X}$, equipped with the operator norm. Let  $\mathcal{K} : \mathcal{X} \mapsto \mathcal{X}$ be a compact linear operator, and let the resolvent set of $\mathcal{K}$ is defined as 
\[
\rho(\mathcal{K}) = \{ z \in \mathbb{C} : (\mathcal{K} - zI)^{-1} \in \mathcal{BL}(\mathcal{X}) \}
\]
and the spectrum of $\mathcal{K}$ is $\sigma(\mathcal{K}) = \mathbb{C} \setminus \rho(\mathcal{K})$. The point spectrum of $\mathcal{K}$ contains all $\lambda \in \sigma(\mathcal{K})$ for which $\mathcal{K} - \lambda I$ is not injective. In such cases, $\lambda$ is referred as an eigenvalue of $\mathcal{K}$. Let $\Gamma$ be a rectifiable, simple, closed, positively oriented curve contained in $\rho(\mathcal{K})$ that separates $\lambda$ from the remainder of the spectrum of $\mathcal{K}$. The spectral projection corresponding to  $\mathcal{K}$ and $\lambda$ is then given by:
\[
E = -\frac{1}{2\pi i} \int_{\Gamma} (\mathcal{K} - zI)^{-1} \, dz \in  \mathcal{BL}(\mathcal{X}),
\]
and the dimension of $\mathcal{R}(E)$, the range of the operator $E$, is referred as the algebraic multiplicity of $\lambda$. If $\lambda$ has finite algebraic multiplicity, it is classified as a spectral value of finite type. In this case, $\lambda$ is an eigenvalue of $\mathcal{K}$. ( See Ahues et al \cite[Proposition 1.31]{Ahues}.) Furthermore, if the algebraic multiplicity of $\lambda$ is equal to one, $\lambda$ is called a simple eigenvalue of $\mathcal{K}$. For any nonzero subspaces $Y$ and $Z$ of $\mathcal{X}$, define 
\[
\delta(Y, Z) = \sup \{d(y, Z) : y \in Y, \norm{y}_\infty = 1\},
\]
where $d(y, Z) = \inf_{z \in Z} \norm{y - z}$. The gap between $Y$ and $Z$ is then given by  
\[
\hat{\delta}(Y, Z) = \max\{\delta(Y, Z), \delta(Z, Y)\}.
\]	
Consider a sequence of operators $\{\mathcal{K}_n\} \subset \mathcal{BL}(\mathcal{X})$ that converges to $\mathcal{K}$ in a collectively compact manner. The following results are taken from Osborn \cite{Osborn}:  

For sufficiently large $n$,  $\Gamma \subset \rho(\mathcal{K}_n)$ and
$\max \{ \norm{(\mathcal{K}_n - zI)^{-1}} : z \in \Gamma \} < \infty$,
the spectral projection corresponding to $\mathcal{K}_n$, given by  
\[
E_n = -\frac{1}{2\pi i} \int_{\Gamma} (\mathcal{K}_n - zI)^{-1} \, dz,
\]
has rank $m$ as suppose the spectrum of $\mathcal{K}_n$ consists of $m$ eigenvalues $\lambda_{n,1}, \lambda_{n,2}, \ldots, \lambda_{n,m}$, counted according to their algebraic multiplicities. Define the arithmetic mean of these eigenvalues as  
\[
\hat{\lambda}_n = \frac{1}{m} \sum_{j=1}^m \lambda_{n,j}.
\]	
Consider the following result from Osborn \cite{Osborn}:
\begin{theorem} \label{th1}
	For sufficiently large $n$,  
	\[
	\hat{\delta}\left(\mathcal{R}(E), \mathcal{R}(E_n)\right) \leq C \norm{\left(\mathcal{K} - \mathcal{K}_n\right)\mathcal{K}|_{\mathcal{R}(E)}} ,
	\]
	where $C$ denotes a generic constant independent of $n$ and $\left( \mathcal{K} - \mathcal{K}_n \right)\mathcal{K}|_{\mathcal{R}(E)}$ represents the restriction of $(\mathcal{K} - \mathcal{K}_n)\mathcal{K}$ to $\mathcal{R}(E)$.
\end{theorem} 
\noindent
Also, we consider a modified result of Osborn \cite[Theorem 2]{Osborn} given in Kulkarni \cite [Theorem 2.2]{RPK6} as follows:
\begin{theorem} \label{th2}
	For sufficiently large $n$,  
	\[
	| \lambda - \hat{\lambda}_n | \leq C \norm{ \mathcal{K}_n(\mathcal{K} - \mathcal{K}_n)\mathcal{K}|_{\mathcal{R}(E)}},
	\]
	where $C$ is a generic constant.
\end{theorem}
\noindent
The above two theorems establish the essential error bounds for the eigenelements.
Here, we will fix some notations to be used throughout the paper:\\
Let $j, k$ be non-negative integer and we set
\begin{equation*}
	D^{(j,k)}\kappa_i(s,t)=\frac{\partial^{j+k}\kappa_i}{\partial s^j \partial t ^k} (s,t), \quad \text{for} \; i=1,2.	
\end{equation*}
Define
\begin{equation*}
	C_1= \max_{0 \le j, k \le {2r+2}}\bigg \{ \sup_{ (s,t) \in \Omega_i} \left| D^{(j,k)} \kappa_i(s,t)\right| : \; \text{for}\; i=1, 2\bigg \},
\end{equation*}  
For a positive integer $\alpha$, if $x \in C^{2r+\alpha}[0,1]$, we define
\begin{equation*}
	\norm{x}_{2r + \alpha, \infty} = \max_{0 \leq i \leq 2r +\alpha} \norm{x^{(i)}}_\infty = \max_{0 \leq i \leq 2r +\alpha} \left\{ \sup_{t \in [0,1]}  \left|x^{(i)}(t)\right| \right\}.
\end{equation*}   
where $x^{(i)}$ is the $i^{th}$ derivative of the function $x$, and
in particular, 
\[
\norm{x}_\infty =  \sup_{t \in [0,1]}  \left|x(t)\right|.
\]
Now, for $x \in C[0,1]$ and $s \in [0,1]$, we have
\[
(\mathcal{K}x)(s) =  \int_{0}^{s} \kappa_1(s, t) x(t)~dt + \int_{s}^{1} \kappa_2(s, t) x(t)~dt.
\]
Using the Leibnitz rule
\begin{equation*}
	(\mathcal{K}x)'(s) = \int_{0}^{s} \frac{\partial \kappa_1}{\partial s}(s, t) x(t)~dt + \int_{s}^{1} \frac{\partial \kappa_2}{\partial s}(s, t) x(t)~dt,
\end{equation*}
\begin{equation*}
	(\mathcal{K}x)''(s) =  \left( \frac{\partial \kappa_1}{\partial s}(s, s) - \frac{\partial \kappa_2}{\partial s}(s, s) \right) x(s) + \int_{0}^{s} \frac{\partial^2 \kappa_1}{\partial s^2}(s, t) x(t)~dt + \int_{s}^{1} \frac{\partial^2 \kappa_2}{\partial s^2}(s, t) x(t)~dt,
\end{equation*}
 which implies
\begin{align*}
	| (\mathcal{K}x)'(s) | &\leq \left( \sup_{(s,t) \in \Omega_1} \left| D^{(1,0)} \kappa_1(s,t) \right|   s  \, + \sup_{(s,t) \in \Omega_2} \left| D^{(1,0)} \kappa_2(s,t) \right| (1-s) \right)  \norm{x}_\infty,
\end{align*}
\begin{align*}
	| (\mathcal{K}x)''(s) | &\leq \left( \sup_{(s,t) \in \Omega_1} \left| D^{(1,0)} \kappa_1(s,t) \right| \, + \sup_{(s,t) \in \Omega_2} \left| D^{(1,0)} \kappa_2(s,t) \right|  \right) \norm{x}_\infty \\
	& + \left( \sup_{(s,t) \in \Omega_1} \left| D^{(2,0)} \kappa_1(s,t) \right|   s  \, + \sup_{(s,t) \in \Omega_2} \left| D^{(2,0)} \kappa_2(s,t) \right| (1-s) \right)  \norm{x}_\infty.
\end{align*}
Then, 
\begin{equation} \label{Eq: 004}
	\norm{(\mathcal{K}x)^{(j)}}_\infty \leq C_1 \norm{x}_\infty, \quad \text{for}~ j=0,1,2.
\end{equation}
Let $x \in C^2[0,1]$, we have
\begin{align*}
	(\mathcal{K}x)^{(3)}(s) &=  2 \left( \frac{\partial^2 \kappa_1}{\partial s^2}(s, s) - \frac{\partial^2 \kappa_2}{\partial s^2}(s, s) \right) x(s) + \left( \frac{\partial \kappa_1}{\partial s}(s, s) - \frac{\partial \kappa_2}{\partial s}(s, s) \right) x'(s) \\ 
	&+ \int_{0}^{s} \frac{\partial^3 \kappa_1}{\partial s^3}(s, t) x(t)~dt + \int_{s}^{1} \frac{\partial^3 \kappa_2}{\partial s^3}(s, t) x(t)~dt,
\end{align*}
and
\begin{align*}
	(\mathcal{K}x)^{(4)}(s) &=  3 \left( \frac{\partial^3 \kappa_1}{\partial s^3}(s, s) - \frac{\partial^3 \kappa_2}{\partial s^3}(s, s) \right) x(s) \\
	&+ 3 \left( \frac{\partial^2 \kappa_1}{\partial s^2}(s, s) - \frac{\partial^2 \kappa_2}{\partial s^2}(s, s) \right) x'(s) + \left( \frac{\partial \kappa_1}{\partial s}(s, s) - \frac{\partial \kappa_2}{\partial s}(s, s) \right) x''(s) \\ 
	&+ \int_{0}^{s} \frac{\partial^4 \kappa_1}{\partial s^4}(s, t) x(t)~dt + \int_{s}^{1} \frac{\partial^4 \kappa_2}{\partial s^4}(s, t) x(t)~dt.
\end{align*}
This implies that, if $x \in C^2[0,1]$, then
\begin{equation*}
	\norm{(\mathcal{K}x)^{(3)}}_\infty \leq C_1 \left( \norm{x}_\infty + \norm{x'}_\infty \right), ~~
	\norm{(\mathcal{K}x)^{(4)}}_\infty \leq C_1 \left( \norm{x}_\infty + \norm{x'}_\infty +  \norm{x''}_\infty \right).
\end{equation*}
Let $r \geq 1$, $\kappa$ be the Green's function-type kernel (as defined in section 2) and $x \in C^{2r}[0,1]$. Then proceeding similarly, we obtain
\begin{equation} \label{Eq: 005}
	\norm{(\mathcal{K}x)^{(2r+1)}}_\infty \leq C_1 \left( \norm{x}_\infty + \norm{x'}_\infty + \cdots + \norm{x^{(2r-1)}}_\infty \right),
\end{equation}
\begin{equation} \label{Eq: 006}
	\norm{(\mathcal{K}x)^{(2r+2)}}_\infty \leq C_1 \left( \norm{x}_\infty + \norm{x'}_\infty + \cdots + \norm{x^{(2r)}}_\infty \right).
\end{equation}


\section{Approximations based on projections}

Denote the uniform partition of $[0,1]$ as
$$
\Delta^{(n)} := \{ 0 =t_0 < t_1 < \cdots <t_n=1\}, \quad \text{where} ~ t_j = \displaystyle{\frac{j}{n}}, \quad j = 0, 1, 2, \ldots, n.
$$ 
Denote $\Delta_{j} = [t_{j-1}, t_j]$. Let $r$ be a non-negative integer, and define
$$
\mathcal{X}_n = \left\{x \in L^\infty[0, 1] : x|_{\Delta_{j}} \text{ is a polynomial of even degree} \leq 2r, \; j=1,2, \ldots, n \right\}. 
$$ 
This is a finite dimensional space with $\dim(\mathcal{X}_n) = n(2r+1)$. Since continuity at the partition points is not imposed, the approximation space $\mathcal{X}_n$ is a subspace of $L^\infty[0,1]$. We consider two projection operators mapping into $\mathcal{X}_n$: the orthogonal projection $P_n$, obtained as the restriction to $L^\infty[0,1]$ of the $L^2$-orthogonal projection onto $\mathcal{X}_n$, and the interpolatory projection $Q_n:\mathcal{C}[0,1]\mapsto\mathcal{X}_n$, defined by interpolation at $2r+1$ points on each subinterval. The projections are discussed in the following subsections:

\subsection{Orthogonal projection}

Let $\mathcal{X} = L^2[0, 1]$ denote the space of real-valued, square-integrable functions defined on the interval $[0, 1]$. The inner product in $L^2[0, 1]$, denoted by $\langle \cdot, \cdot \rangle$, is defined as  
\[
\langle x, y \rangle = \int_0^1 x(t) y(t) \, dt, \quad \forall x, y \in L^2[0, 1].
\]
For \(\eta = 0, 1, \ldots, 2r\), let \(L_{\eta}\) represent the Legendre polynomial of degree \(\eta\) defined on the interval \([-1,1]\). For each \(j = 1, \ldots, n\) and \(\eta = 0, 1, \ldots, 2r\), we define the function \(\psi_{j,\eta}(t)\) as follows:
\[
\psi_{j,\eta}(t) =
\begin{cases}
	\sqrt{\frac{2}{h}} L_{\eta} \left( \frac{2t - t_j - t_{j-1}}{h} \right), & t \in (t_{j-1}, t_j], \\
	0, & \text{otherwise}.
\end{cases}
\]
At the point \(t_0\), we set \(\psi_{1,\eta}(t_0) = \sqrt{\frac{2}{h}} L_{\eta}(-1)\). The supremum norm of \(\psi_{j,\eta}\) is given by
\[
\norm{\psi_{j,\eta}}_{\infty} = \max_{t \in [t_{j-1}, t_j]} |\psi_{j,\eta}(t)| = \sqrt{\frac{2}{h}} \norm{L_{\eta}}_{\infty}.
\]
For functions \(f, g \in C(\Delta_j)\), we define the inner product on the interval \(\Delta_j\) as
\[
\langle f, g \rangle_{\Delta_j} = \int_{t_{j-1}}^{t_j} f(t) g(t) \, dt.
\]
This inner product is symmetric, i.e., \(\langle f, g \rangle_{\Delta_j} = \langle g, f \rangle_{\Delta_j}\). The supremum norm of \(f\) on \(\Delta_j\) is defined as
\[
\norm{f}_{\Delta_j, \infty} = \max_{t \in [t_{j-1}, t_j]} |f(t)|.
\]
Using this, we can bound the inner product as
\[
\langle f, g \rangle_{\Delta_j} \leq \norm{f}_{\Delta_j, \infty} \norm{g}_{\Delta_j, \infty} h. 
\]
It is easy to see that the set \(\{\psi_{j,\eta}: j = 1, \ldots, n, \eta = 0, \ldots, 2r\}\) forms an orthonormal basis for the space \(\mathcal{X}_n\). Let \(\mathcal{P}_{2r, \Delta_j}\) denote the space of polynomials of degree at most \(2r\) on \(\Delta_j\). The orthogonal projection \(P_{n,j} : C(\Delta_j) \mapsto \mathcal{P}_{2r, \Delta_j}\) is defined by
\[
P_{n,j} x = \sum_{\eta=0}^{2r} \langle x, \psi_{j,\eta} \rangle_{\Delta_j} \psi_{j,\eta}.
\]
This projection satisfies the following properties:
\[
\langle P_{n,j} x, y \rangle_{\Delta_j} = \langle x, P_{n,j} y \rangle_{\Delta_j}, \quad P_{n,j}^2 = P_{n,j}, \quad \text{and} \quad P_{n,j} P_{n,i} = 0 \quad \text{for } i \neq j.
\]
Additionally, we have the bound as
\[
\norm{P_{n,j} x}_{\Delta_j, \infty} \leq \sum_{\eta=0}^{2r} \langle x, \psi_{j,\eta} \rangle_{\Delta_j} \|\psi_{j,\eta}\|_{\Delta_j, \infty} \leq 2 \left( \sum_{\eta=0}^{2r} \norm{L_{\eta}}_{\infty}^2 \right) \norm{x}_{\infty}.
\]
Thus, an orthogonal projection \(P_n : C[0,1] \mapsto \mathcal{X}_n\) is defined as
\begin{equation} \label{Eq: 007}
	P_n x = \sum_{j=1}^{n} P_{n,j} x.
\end{equation}
By applying the Hahn-Banach extension theorem, following the approach in Atkinson et al. \cite{KEA-IG-IS}, the projection \(P_n\) can be extended to the space \(L^\infty[0,1]\). This extended projection satisfies \(P_n^2 = P_n\) and 
\begin{equation} \label{Eq: 008}
	\norm{P_n} \leq 2 \sum_{\eta=0}^{2r} \norm{L_{\eta}}_{\infty}^2 =C_2. 
\end{equation}
Finally, if \(x \in C^{2r+1}(\Delta_j)\), we have the standard estimate from Chatelin-Lebbar \cite{Cha-Leb2} as
\begin{equation} \label{Eq: 009}
	\norm{(I - P_{n,j}) x}_{\Delta_j, \infty} \leq C_3 \norm{x^{(2r+1)}}_{\Delta_j, \infty} h^{2r+1},
\end{equation}
where $C_3$ is a constant independent of $h$. Moreover, if $x \in C^{\beta}[0,1]$ with $\beta = \min\{\alpha,\, 2r+1\}$, then one can see that
\begin{equation} \label{Eq: 0010}
	\norm{(I - P_{n})x}_\infty \leq C_3 \norm{x^{(\beta)}}_\infty h^{\beta}.
\end{equation}

\subsection{Interpolatory projection}

Let $\mathcal{X} = C[0,1]$ be the space of continuous real valued function on the interval $[0, 1]$,  equipped with the supremum norm defined as \(\norm{x}_\infty =  \sup_{t \in [0,1]}  \left|x(t)\right|\). In each sub-interval $[t_{j-1}, t_j]$, consider $2r+1$ equally spaced interpolation points as
\begin{equation*}
	\tau_j^i = t_{j-1} + \frac{ih}{2r}, \quad i = 0, 1, \ldots, 2r  \; \text{ for } \; r \geq 1, \quad \text{and} \quad \tau_j = \frac{t_{j-1} + t_j}{2} \; \text{ for }  \; r = 0.
\end{equation*}
Note that these interpolation nodes are equidistant in each subinterval. The interpolatory operator $Q_n : \mathcal{X} \mapsto \mathcal{X}_n $ is defined as
\begin{equation} \label{Eq: 0011}
	Q_nx(\tau_j^i) = x(\tau_j^i), \quad i = 0, 1, 2, \ldots, 2r; ~ j = 1, 2, \ldots, n.
\end{equation}
Note that \(Q_nx \to x ~ \text{for all} \; x \in C[0,1]\). By employing the Hahn-Banach extension theorem, \(Q_n\) can be extended to \(L^\infty[0,1]\), making \(Q_n : L^\infty[0,1] \mapsto \mathcal{X}_n\) a projection. Define the polynomial 
\begin{eqnarray*}
	\Psi_j(t) & = & \prod_{i=0}^{2r} \left(t - \tau_j^i \right), \quad j = 1, 2, \ldots, n.
\end{eqnarray*}
As \(Q_{n,j}x\) (where \(Q_{n,j}\) is \(Q_{n}\) restricted on \(\Delta_j\)) interpolates \(x\) at the points \(\tau_j^0, \tau_j^1, \ldots, \tau_j^{2r}\), we have
\begin{eqnarray*}
	x(t) - Q_{n,j}x(t) & = & \Psi_j(t) \left[\tau_j^0, \tau_j^1, \ldots, \tau_j^{2r}, t\right]x, \quad t \in \Delta_j,
\end{eqnarray*}
where $\left[\tau_j^0, \tau_j^1, \ldots, \tau_j^{2r}, t\right]x$ is the $(2r+1)^{th}$ divided difference of $x \in C[0, 1]$ at $\tau_j^0, \ldots, \tau_j^{2r}$ defined as
\begin{eqnarray*}
	\left[\tau_j^0, \ldots, \tau_j^{2r}\right]x &=& \frac{\left[\tau_j^1, \ldots, \tau_j^{2r}\right]x - \left[\tau_j^0, \ldots, \tau_j^{2r-1}\right]x}{\tau_j^{2r} - \tau_j^0}.
\end{eqnarray*}
If \(x \in C^{2r+1}(\Delta_j)\), a standard result (see \cite{KEA-IG-IS}) implies that
\begin{equation*} 
	\norm{(I-Q_{n,j})x}_{\Delta_j, \infty}  \leq C_4 \norm{x^{(2r+1)}}_{\Delta_j, \infty} h^{2r+1},
\end{equation*}
where $C_4$ is a constant independent of $h$. Indeed, if $x \in C^{\beta}[0,1]$ with $\beta = \min\{\alpha,\, 2r+1\}$, then
\begin{equation} \label{Eq: 0012}
	\norm{(I - Q_{n})x}_\infty \leq C_4 \norm{x^{(\beta)}}_\infty h^{\beta}.
\end{equation}
\subsection{Estimates in the methods of approximation}

The classical projection method for approximating the eigenvalue problem \eqref{Eq: 001} is obtained by replacing $\mathcal{K}_n$ by $\pi_n \mathcal{K}$ in \eqref{Eq: 003} as
\begin{equation} \label{Eq: 0013}
	\pi_n \mathcal{K} \varphi_n = \lambda_n \varphi_n,  ~~\lambda_n \in \mathbb{C} \setminus \{0\}, ~ \varphi_n \in \mathcal{X}_n, ~\norm{\varphi_n}_\infty =1.
\end{equation}
One step iteration refine the eigenvector, proposed by Sloan \cite{Sloan1}, is given by
\begin{equation} \label{Eq: 0014}
	\varphi_n^S = \frac{1}{\lambda_n} \mathcal{K} \varphi_n,
\end{equation}
which satisfies the condition
\[
\mathcal{K} \pi_n \varphi_n^S = \lambda_n \varphi_n^S.
\]
In the modified projection method introduced by Kulkarni \cite{RPK5}, the operator $\mathcal{K}_n$ is replaced with a finite-rank approximation, defined as
\begin{equation} \label{Eq: 0015}
	\mathcal{K}_n^M = \pi_n \mathcal{K} +\mathcal{K} \pi_n - \pi_n \mathcal{K} \pi_n.
\end{equation}
The approximation \eqref{Eq: 003} now requires finding $\lambda_n^M \in \mathbb{C} \setminus \{0\}$ and $\varphi_n^M \in \mathcal{X}$ such that $\norm{\varphi_n^M}_\infty = 1$ and
\begin{equation} \label{Eq: 0016}
	\mathcal{K}_n^M \varphi_n^M = \lambda_n^M \varphi_n^M. 
\end{equation}
The corresponding iterated eigenvector for the modified method is defined as
\begin{equation} \label{Eq: 0017}
	\widetilde{\varphi}_n^M = \frac{1}{\lambda_n^M} \mathcal{K} \varphi_n^M. 
\end{equation}

Let $\{\mathcal{K}_n\}$ be a sequence of bounded linear operators on $\mathcal{X}$ converging to $\mathcal{K}$ in the collectively compact sense. We consider the
standard and modified projection-based approximations defined by
$\mathcal{K}_n=\pi_n\mathcal{K}$ and $\mathcal{K}_n=\mathcal{K}_n^{M}$, respectively.
Let $\lambda$ be a nonzero eigenvalue of $\mathcal{K}$ with algebraic multiplicity $m$,
and let $\hat{\lambda}_n$ and $\hat{\lambda}_n^{M}$ denote the arithmetic means of the
$m$ eigenvalues of $\pi_n\mathcal{K}$ and $\mathcal{K}_n^{M}$, respectively, lying inside
a simple closed contour $\Gamma$ that encloses $\lambda$. Let $E$, $E_n$, and $E_n^{M}$ denote the corresponding spectral projections associated with $\Gamma$. The associated spectral subspaces are denoted by $\mathcal{R}(E)$, $\mathcal{R}(E_n)$, and
$\mathcal{R}(E_n^{M})$. Then, the following results can be obtained from Theorem \ref{th1} and Theorem \ref{th2} as given in Bouda et al. \cite{HB-CA-ZEA-KK}:\\
For any $\varphi_n \in \mathcal{R}(E_n)$, define $\varphi_n^{S} = \mathcal{K}\varphi_n$. Then, for all
sufficiently large $n$,
\begin{equation} \label{Eq: 0018}
	\norm{\varphi_n - E \varphi_n} \leq \hat{\delta} \big(\mathcal{R}(E), \mathcal{R}(E_n)\big) \leq C \norm{(\mathcal{K} - \pi_n \mathcal{K})\mathcal{K}|_{\mathcal{R}(E)}},
\end{equation}
\begin{equation} \label{Eq: 0019}
	\norm{\varphi_n^S - E \varphi_n^S} \leq \delta\big(\mathcal{R}(E), \mathcal{K} \mathcal{R}(E_n)\big) \leq C \norm{\mathcal{K}(\mathcal{K} - \pi_n \mathcal{K})\mathcal{K}|_{\mathcal{R}(E)}},
\end{equation}
and 
\begin{equation} \label{Eq: 0020}
	| \lambda - \hat{\lambda}_n | \leq C \norm{ \mathcal{K}(\mathcal{K} - \pi_n \mathcal{K})\mathcal{K}|_{\mathcal{R}(E)}}.
\end{equation}

\medskip
\noindent
Also, for any $\varphi_n^{M} \in \mathcal{R}(E_n^{M})$, define $\widetilde{\varphi}_n^{M} = \mathcal{K}\varphi_n^{M}$. Then, for all sufficiently large $n$, 
\begin{equation} \label{Eq: 0021}
	\norm{\varphi_n^M - E \varphi_n^M} \leq \hat{\delta} \big(\mathcal{R}(E), \mathcal{R}(E_n^M)\big) \leq C \norm{(\mathcal{K} - \mathcal{K}_n^M)\mathcal{K}|_{\mathcal{R}(E)}},
\end{equation}
\begin{equation} \label{Eq: 0022}
	\norm{\widetilde{\varphi}_n^M - E \widetilde{\varphi}_n^M} \leq \delta\big(\mathcal{R}(E), \mathcal{K} \mathcal{R}(E_n^M)\big) \leq C \norm{\mathcal{K}(\mathcal{K} - \mathcal{K}_n^M)\mathcal{K}|_{\mathcal{R}(E)}},
\end{equation}
and 
\begin{equation} \label{Eq: 0023}
	| \lambda - \hat{\lambda}_n^M | \leq C \norm{ \mathcal{K}(\mathcal{K} -  \mathcal{K}_n^M)\mathcal{K}|_{\mathcal{R}(E)}}.
\end{equation}


\section{Convergence orders in the approximations}

In this section, we focus on the case where $\lambda$ is a simple eigenvalue of $\mathcal{K}$. 
For the projection methods considered earlier, we take $\pi_n = P_n$ for convenience, 
in which case the scheme reduces to the Galerkin method and its variants. 
Similarly, setting $\pi_n = Q_n$ yields the collocation method and its variants. 
Accordingly, we establish our main results for both the Galerkin and collocation methods..

\subsection{Orthogonal projection based approximations}

In this subsection, we present essential findings that will be instrumental in determining approximate eigenelements through orthogonal projection.
\begin{proposition} \label{pro1}
	For $r \geq 1$, if $x \in C^{2r+1}[0,1]$, then
	\begin{eqnarray*}
		\norm{\mathcal{K}(I-P_n)x }_\infty \leq C_5 \norm{x^{(2r+1)}}_{ \infty} h^{2r + 3},
	\end{eqnarray*}
	where $C_5$ is a constant independent of $h$. Also,
	\begin{eqnarray*} 
		\norm{\mathcal{K}(I-P_n)x }_\infty \leq C_6 \norm{x}_\infty h^2,
	\end{eqnarray*}
	for some constant $C_6$ independent of $h$.
\end{proposition}
\begin{proof} First we prove the case \(r \geq 1\).
	Let \(s \in [0,1]\) and define \(\kappa_s(t) = \kappa(s,t)\). Then, consider
	\begin{align} \label{Eq: 0024}
		\mathcal{K}(I - P_n)x(s) &= \int_{0}^{1} \kappa_s(t) (I - P_n) x(t) \, dt \notag \\
		&= \sum_{j=1}^{n} \int_{t_{j-1}}^{t_j} \kappa_s(t) (I - P_{n,j}) x(t) \, dt \notag \\
		&= \sum_{j=1}^{n} \langle \kappa_s, (I - P_{n,j}) x \rangle_{\Delta_j}  \notag \\
		&= \sum_{j=1}^{n} \langle (I - P_{n,j}) \kappa_s, (I - P_{n,j}) x \rangle_{\Delta_j}.
	\end{align}
	We analyze the two cases separately:\\
	\noindent
	Case I: Let $s \in (t_{i-1}, t_i)$, for $i = 1,2,\ldots,n$. Thus, the above expression can be rewritten as
	\begin{equation} \label{Eq: 0025}
		\mathcal{K}(I - P_n)x(s) = \sum_{\substack{j=1 \\  j\ne i} }^{n} \langle (I - P_{n,j}) \kappa_s, (I - P_{n,j}) x \rangle_{\Delta_j} + \langle (I - P_{n,i}) \kappa_s, (I - P_{n,i}) x \rangle_{\Delta_i}.
	\end{equation}
	For $j \neq i$, we note that both $\kappa_s, x \in C^{2r+1}(\Delta_j)$. Thus, applying the estimate \eqref{Eq: 009}, we have
	\begin{equation} \label{Eq: 0026}
		\bigg|\sum_{\substack{j=1 \\  j\ne i} }^{n} \langle (I - P_{n,j}) \kappa_s, (I - P_{n,j}) x \rangle_{\Delta_j}\bigg|\leq (C_3)^2 C_1 \norm{x^{(2r+1)}}_\infty (n-1) h^{4r+3}.
	\end{equation}
	Now, for the case $j = i$, observe that $\kappa_s$ is only continuous on $\Delta_i$. We define a constant function as
	\begin{equation*}
		\gamma_i(t) = \kappa_s(s), \quad t \in \Delta_i.
	\end{equation*}
	Then, we rewrite the term as
	\begin{equation*}
		\langle (I - P_{n,i}) \kappa_s, (I - P_{n,i}) x \rangle_{\Delta_i} = \langle \kappa_s - \gamma_i, (I - P_{n,i}) x \rangle_{\Delta_i}.
	\end{equation*}
	For $t \in [t_{i-1}, t_i]$, using the mean value theorem, we have
	\begin{equation*}
		\kappa_s(t) - \gamma_i(t) = \begin{cases}
			D^{(0,1)} \kappa_1(s, \nu_1) (t - s), & \nu_1 \in (t, s), \\
			D^{(0,1)} \kappa_2(s, \nu_2) (t - s), & \nu_2 \in (s, t).
		\end{cases}
	\end{equation*}
	Further using the bound \eqref{Eq: 009}, we get
	\begin{equation} \label{Eq: 0027}
		|\langle (I - P_{n,i}) \kappa_s, (I - P_{n,i}) x \rangle_{\Delta_i}| \leq C_3 C_1  \norm{x^{(2r+1)}}_\infty h^{2r+3}.
	\end{equation}
	Combining the inequalities \eqref{Eq: 0025}, \eqref{Eq: 0026} and \eqref{Eq: 0027}, we conclude
	\begin{equation} \label{Eq: 0028}
		|\mathcal{K}(I - P_n)x(s)| \leq (C_3)^2 C_1 \norm{x^{(2r+1)}}_\infty h^{2r+3}.
	\end{equation}
	\noindent
	Case II: Whenever $s = t_i$ for some $i = 0,1,\ldots,n$. Since $\kappa_s \in C^{2r+1}(\Delta_j)$ for each $j = 1, \ldots, n$ and given that $x \in C^{2r+1}[0,1]$, we use inequality \eqref{Eq: 009} to obtain
	\begin{align*} 
		\max_{0 \leq i \leq n} |\mathcal{K}(I - P_n)x(t_i)| 
		&\leq \sum_{j=1}^{n} \norm{(I - P_{n,j})\kappa_s}_{\Delta_j, \infty} \norm{(I - P_{n,j})x}_{\Delta_j, \infty} h \\
		&\leq (C_3)^2 C_1 \norm{x^{(2r+1)}}_\infty h^{4r+2}.
	\end{align*}
	Finally, combining this estimate with \eqref{Eq: 0028}, the desired result follow as
	\begin{equation*}
		\norm{\mathcal{K}(I-P_n)x }_\infty \leq C_5 \norm{x^{(2r+1)}}_{ \infty} h^{2r + 3},
	\end{equation*}
	where $C_5 = (C_3)^2 C_1 $ is a constant independent of $h$.\\
	To establish the other bound of the proposition (for \(r = 0\)), let us consider $x \in C[0,1]$. Suppose $s = t_i$ for some index $i$. Then,
	\begin{equation} \label{Eq: 0029}
		\left| \mathcal{K}(I - P_n)x(s) \right| \leq C_1 C_3 \norm{x}_\infty h^{2r+1}.
	\end{equation}
	Now, for $s \in (t_{i-1}, t_i)$, we express
	\begin{equation*}
		\mathcal{K}(I - P_n)x(s) =\sum_{\substack{j=1 \\  j\ne i} }^{n} \langle (I - P_{n,j}) \kappa_s,  x \rangle_{\Delta_j} + \langle (I - P_{n,i}) (\kappa_s - \gamma_i),x \rangle_{\Delta_i},
	\end{equation*}
	which leads to the estimate
	\begin{equation*}
		\left| \mathcal{K}(I - P_n)x(s) \right| \leq (1 +C_2 + C_3) C_1 \norm{x}_\infty h^2.
	\end{equation*}
	By combining \eqref{Eq: 0029} with the above result, we arrive at
	\begin{equation*}
		\norm{\mathcal{K}(I - P_n)x}_\infty \leq C_6 \norm{x}_\infty h^2,
	\end{equation*}
	where $C_6 = (1 +C_2 + C_3) C_1$ is a constant. This completes the proof.
\end{proof}

\begin{lemma} \label{lem1}
	Let $ \mathcal{K}$ be the linear integral operator defined by \eqref{Eq: 002} with a Green's function-type kernel. Let $P_n$ be the orthogonal projection onto the approximating space $\mathcal{X}_n$ defined by \eqref{Eq: 007}. Then 
	\begin{eqnarray*} 
		\norm{(\mathcal{K} - P_n \mathcal{K}) \mathcal{K} |_{\mathcal{R}(E)} } = 	\begin{cases}
			O\left(h^{2r+1}\right) & \text{for } r \geq 1, \\
			O\left(h\right) & \text{for } r = 0,
		\end{cases}
	\end{eqnarray*}
	and	
	\begin{eqnarray*}
		\norm{ \mathcal{K} (\mathcal{K} - P_n \mathcal{K}) \mathcal{K} |_{\mathcal{R}(E)} } =
		\begin{cases}
			O\left(h^{2r+3}\right) & \text{for } r \geq 1, \\
			O\left(h^{2}\right) & \text{for } r = 0.
		\end{cases}
	\end{eqnarray*}
\end{lemma}
\begin{proof} Note that 	
	\begin{equation*}
		\norm{(\mathcal{K} - P_n \mathcal{K}) \mathcal{K} |_{\mathcal{R}(E)} } = \sup_{\varphi \in \mathcal{R}(E)}\left\{ \norm{(\mathcal{K} - P_n \mathcal{K}) \mathcal{K} \varphi }_\infty : \norm{\varphi}_\infty =1 \right\}.
	\end{equation*}
	For \(\varphi \in \mathcal{R}(E)\), assume that \(y= \mathcal{K}\varphi\) and so \(y \in \mathcal{R}(E)\). For \(r \geq 1\), let us consider \( y \in C^{2r+1}[0,1]\) so that \(\mathcal{K}y \in C^{2r+1}[0,1]\) and therefore using \eqref{Eq: 0010}, we have
	\begin{equation*}
		\norm{(\mathcal{K} - P_n \mathcal{K}) \mathcal{K} \varphi }_\infty = \norm{(I - P_n ) \mathcal{K} y }_\infty \leq C_3 \norm{\left(\mathcal{K}y\right)^{(2r+1)}}_\infty h^{2r+1}.
	\end{equation*}
	From inequality \eqref{Eq: 005}, it follows that  
	\begin{align*}
		\norm{\left(\mathcal{K}y\right)^{(2r+1)}}_\infty &\leq C_1 \left(\norm{y}_\infty + \norm{y'}_\infty + \ldots + \norm{y^{(2r-1)}}_\infty\right) \\  
		&= C_1 \left(\norm{\mathcal{K}\varphi}_\infty + \norm{(\mathcal{K}\varphi)'}_\infty + \ldots + \norm{(\mathcal{K}\varphi)^{(2r-1)}}_\infty\right).
	\end{align*}  
	By utilizing inequalities \eqref{Eq: 004} and \eqref{Eq: 005}, we conclude that the  quantity \( \norm{\left(\mathcal{K}y\right)^{(2r+1)}}_\infty \) is independent of \( h \). Therefore, we obtain 
	\begin{equation*}
		\norm{(\mathcal{K} - P_n \mathcal{K}) \mathcal{K} \varphi }_\infty = O\left(h^{2r+1}\right),
	\end{equation*}
	which leads to
	\begin{equation*}
		\norm{(\mathcal{K} - P_n \mathcal{K}) \mathcal{K} |_{\mathcal{R}(E)} } = 	O\left(h^{2r+1}\right).
	\end{equation*}
	When \(r = 0\), it is easy to see that 
	\begin{equation*}
		\norm{(\mathcal{K} - P_n \mathcal{K}) \mathcal{K} |_{\mathcal{R}(E)} } = O\left(h \right).
	\end{equation*}
	Now, consider the second bound of the lemma as
	\begin{equation*}
		\norm{ \mathcal{K} (\mathcal{K} - P_n \mathcal{K}) \mathcal{K} |_{\mathcal{R}(E)} } = \sup_{\varphi \in \mathcal{R}(E)}\left\{ \norm{\mathcal{K}(\mathcal{K} - P_n \mathcal{K}) \mathcal{K} \varphi }_\infty : \norm{\varphi}_\infty =1 \right\}.
	\end{equation*}
	Assume that \(\varphi \in C^{2r+1}[0,1]\), which implies \(y \in C^{2r+1}[0,1]\) and so \(\mathcal{K}y \in C^{2r+1}[0,1]\). Using Proposition \ref{pro1} for \(r \geq 1\), we have
	\begin{equation*}
		\norm{\mathcal{K}(\mathcal{K} - P_n \mathcal{K}) \mathcal{K} \varphi }_\infty = \norm{\mathcal{K}(I - P_n ) \mathcal{K} y }_\infty \leq C_5 \norm{(\mathcal{K} y)^{(2r+1)}}_{ \infty} h^{2r + 3}.
	\end{equation*}
	As discussed earlier, the quantity \( \norm{(\mathcal{K} y)^{(2r+1)}}_{\infty} \) is independent of \( h \) and therefore
	\begin{equation*}
		\norm{\mathcal{K}(\mathcal{K} - P_n \mathcal{K}) \mathcal{K} \varphi }_\infty =  O\left(h^{2r + 3}\right),
	\end{equation*}
	which implies that
	\begin{equation*}
		\norm{ \mathcal{K} (\mathcal{K} - P_n \mathcal{K}) \mathcal{K} |_{\mathcal{R}(E)} } =  O\left(h^{2r + 3}\right).
	\end{equation*}
	When \(r=0\), using the Proposition \ref{pro1}, we have 
	\begin{equation*}
		\norm{\mathcal{K}(I - P_n ) \mathcal{K} y }_\infty \leq C_6 \norm{\mathcal{K} y}_{ \infty} h^{2}.
	\end{equation*}
	From estimate \eqref{Eq: 004}, we see that
	\begin{equation*}
		\norm{\mathcal{K} y}_{ \infty} \leq C_1 \norm{y}_\infty = C_1 \norm{\mathcal{K}\varphi}_\infty \leq(C_1)^2  \norm{\varphi}_\infty. 
	\end{equation*} 
	Therefore, we get
	\begin{equation*}
		\norm{ \mathcal{K} (\mathcal{K} - P_n \mathcal{K}) \mathcal{K} |_{\mathcal{R}(E)} } = O\left(h^{2}\right),
	\end{equation*}
	which completes the proof.
	
\end{proof}
\begin{lemma} \label{lem2}
	Let $ \mathcal{K}$ be the linear integral operator defined by \eqref{Eq: 002} with a Green's function-type kernel. Let $P_n$ be the orthogonal projection onto the approximating space $\mathcal{X}_n$ defined by \eqref{Eq: 007}. Then
	\begin{eqnarray*}
		\norm{(\mathcal{K} - \mathcal{K}_n^M) \mathcal{K} |_{\mathcal{R}(E)}} =
		\begin{cases}
			O\left(h^{2r+3} \right) & \text{for } r \geq 1, \\
			O\left(h^{3} \right) & \text{for } r = 0,
		\end{cases}
	\end{eqnarray*}
	\begin{eqnarray*}
		\norm{\mathcal{K} (\mathcal{K} - \mathcal{K}_n^M) \mathcal{K} |_{\mathcal{R}(E)}} =
		\begin{cases}
			O\left(h^{2r+5} \right) & \text{for } r \geq 1, \\
			O\left(h^{4} \right) & \text{for } r = 0,
		\end{cases}
	\end{eqnarray*}
	where $\mathcal{K}_n^M= P_n \mathcal{K} + \mathcal{K} P_n - P_n \mathcal{K} P_n.$
\end{lemma}
\begin{proof}
	For any \(\varphi \in \mathcal{R}(E)\) we have  \( \mathcal{K}\varphi \in \mathcal{R}(E)\). Then 
	\begin{align*}
		\norm{(\mathcal{K} - \mathcal{K}_n^M) \mathcal{K} |_{\mathcal{R}(E)}} &= \sup_{\varphi \in \mathcal{R}(E)}\left\{ 	\norm{(\mathcal{K} - \mathcal{K}_n^M) \mathcal{K}\varphi}_\infty : \norm{\varphi}_\infty =1 \right\} \\
		&= \sup_{\varphi \in \mathcal{R}(E)}\left\{ \norm{(I - P_n) \mathcal{K}(I - P_n)\mathcal{K}\varphi}_\infty : \norm{\varphi}_\infty =1 \right\}.
	\end{align*}
	For \( r \geq 1 \), we write  
	\begin{equation*}
		\norm{(I - P_n) \mathcal{K}(I - P_n)\mathcal{K}\varphi}_\infty \leq\norm{I - P_n} \norm{\mathcal{K}(I - P_n)\mathcal{K}\varphi}_\infty.
	\end{equation*}
	From \eqref{Eq: 008}, it follows that \( \norm{I - P_n} \leq 1 +C_2 \). Let \( \varphi \in C^{2r+1}[0,1] \), which implies \( \mathcal{K} \varphi \in C^{2r+1}[0,1] \). Moreover, applying Proposition \ref{pro1}, we have
	\begin{equation*}
		\norm{(I - P_n) \mathcal{K}(I - P_n)\mathcal{K}\varphi}_\infty \leq (1 +C_2)C_5 \norm{(\mathcal{K}\varphi)^{(2r+1)}}_\infty h^{2r+3}.
	\end{equation*}
	Following inequality \eqref{Eq: 005}, we observe that the bound \(\norm{ \left(\mathcal{K}\varphi\right)^{(2r+1)}}_\infty \) does not depend on \(h\). Therefore, 
	\begin{equation} \label{Eq: 0030}
		\norm{(I - P_n) \mathcal{K}(I - P_n)\mathcal{K}\varphi}_\infty = O\left(h^{2r+3}\right).
	\end{equation}
	In the case of \(r=0\), we use \eqref{Eq: 0010} to obtain
	\begin{equation} \label{Eq: 0031}
		\norm{(I - P_n) \mathcal{K}(I - P_n)\mathcal{K}\varphi}_\infty \leq C_3 \norm{\left( \mathcal{K}(I - P_n)y\right)'}_\infty h.
	\end{equation}
	Now, define 
	\begin{equation*}
		\ell_s(t) = \ell(s,t) = \begin{cases}
			D^{(1,0)} \kappa_1(s, t) , & s,t \in \Omega_1, \\
			D^{(1,0)} \kappa_2(s, t) , & s,t \in \Omega_2.
		\end{cases}
	\end{equation*}
	Consider 
	\begin{align*} 
		(\mathcal{K}(I - P_n)y)'(s) &= \int_{0}^{1} \ell_s(t) (I - P_n) y(t) \, dt \notag \\
		&= \sum_{j=1}^{n} \int_{t_{j-1}}^{t_j} \ell_s(t) (I - P_{n,j}) y(t) \, dt \notag \\
		&= \sum_{j=1}^{n} \langle (I - P_{n,j}) \ell_s, (I - P_{n,j}) y \rangle_{\Delta_j}.
	\end{align*}
	Let \( s \in [t_{i-1}, t_i]\), for \(i=1,2,\ldots,n\). Then, we can write
	\begin{equation} \label{Eq: 0032}
		(\mathcal{K}(I - P_n)y)'(s) = \sum_{\substack{j=1 \\  j\ne i} }^{n} \langle (I - P_{n,j}) \ell_s, (I - P_{n,j}) y \rangle_{\Delta_j} + \langle (I - P_{n,i}) \ell_s, (I - P_{n,i}) y \rangle_{\Delta_i}.
	\end{equation}
	For \( j \neq i \), we observe that both \( \ell_s \in C^{2r+1}(\Delta_j) \) and \( x \in C^{2r+1}(\Delta_j)\). Hence, utilizing the estimate \eqref{Eq: 009}, we obtain  
	\begin{equation} \label{Eq: 0033}  
		\bigg|\sum_{\substack{j=1 \\ j\ne i} }^{n} \langle (I - P_{n,j}) \ell_s, (I - P_{n,j}) y \rangle_{\Delta_j} \bigg| \leq (C_3)^2 \norm{\ell_s}_\infty \norm{y}_\infty (n-1) h^{3} \leq  (C_3 C_1)^2 \norm{\varphi}_\infty (n-1) h^{3},
	\end{equation}  
	where \(\norm{\ell_s}_\infty = C_1\) by definition. Also, whenever \(j = i\), \(\ell_s\) is continuous on \( [t_{i-1}, t_i]\), for \(i=1,2,\ldots,n\). Thus, following the proof similar to Proposition \ref{pro1}, we get
	\begin{equation*}
		|\langle (I - P_{n,i}) \ell_s, (I - P_{n,i}) y \rangle_{\Delta_i}| = O\left(h^{2} \right).
	\end{equation*}
	Applying the above estimate along with inequality \eqref{Eq: 0033} in \eqref{Eq: 0032}, we have
	\begin{equation*}
		\norm{	(\mathcal{K}(I - P_n)y)'}_\infty = O\left(h^{2} \right)
	\end{equation*}
	Using above estimate in \eqref{Eq: 0031}, we obtain
	\begin{equation*}
		\norm{(I - P_n) \mathcal{K}(I - P_n)\mathcal{K}\varphi}_\infty = O\left(h^{3} \right).
	\end{equation*}
	The above estimate and \eqref{Eq: 0030} establish the first bound of the lemma.\\
	Now, consider the other bound of lemma as
	\begin{equation*}
		\norm{\mathcal{K}(\mathcal{K} - \mathcal{K}_n^M) \mathcal{K} |_{\mathcal{R}(E)}} 
		= \sup_{\varphi \in \mathcal{R}(E)}\left\{\norm{ \mathcal{K} (I - P_n) \mathcal{K}(I - P_n)\mathcal{K}\varphi}_\infty : \norm{\varphi}_\infty =1 \right\}.
	\end{equation*}
	We can write it as follows
	\begin{equation*}
		\norm{\mathcal{K}(I - P_n) \mathcal{K}(I - P_n)\mathcal{K}\varphi}_\infty \leq\norm{\mathcal{K}(I - P_n)} \norm{\mathcal{K}(I - P_n)\mathcal{K}\varphi}_\infty.
	\end{equation*}
	For any \(x \in C[0,1]\), consider
	\begin{equation*}
		\norm{\mathcal{K}(I - P_n)} = \sup_{x \in C[0,1]}\left\{ \norm{\mathcal{K}(I - P_n)x}_\infty : \norm{x}_\infty =1 \right\}.
	\end{equation*}
	Using Proposition \ref{pro1}, we obtain
	\begin{equation*}
		\norm{\mathcal{K}(I - P_n)} = O\left(h^{2} \right).
	\end{equation*}
	For \(r \geq 1\), let \(\varphi \in C^{2r+1}[0,1]\), it follows that \(\mathcal{K}\varphi \in C^{2r+1}[0,1]\) and using Proposition \ref{pro1}, we have
	\begin{equation*}
		\norm{\mathcal{K}(I - P_n)\mathcal{K}\varphi}_\infty = O\left(h^{2r +3}\right).
	\end{equation*}
	Combining the above two estimates, we obtain 
	\begin{equation}  \label{Eq: 0034}
		\norm{\mathcal{K}(I - P_n) \mathcal{K}(I - P_n)\mathcal{K}\varphi}_\infty = 	O\left(h^{2r+5}\right).
	\end{equation}
	When \(r=0\), using Proposition \ref{pro1}, we obtain
	\begin{equation*}
		\norm{\mathcal{K}(I - P_n)\mathcal{K}\varphi}_\infty = O\left(h^{2}\right).
	\end{equation*}
	Therefore,
	\begin{equation*} 
		\norm{\mathcal{K}(I - P_n) \mathcal{K}(I - P_n)\mathcal{K}\varphi}_\infty = 	O\left(h^{4}\right).
	\end{equation*}
	The above estimate along with \eqref{Eq: 0034} establish the desired second bound of the lemma, thereby completing the proof.
\end{proof}
Combining the results \eqref{Eq: 0018}, \eqref{Eq: 0019}, \eqref{Eq: 0020} and Lemma \ref{lem1}, we obtain the following convergence rates for the approximation of eigenvalues and spectral subspaces with Galerkin and iterated Galerkin methods.
\begin{theorem} \label{tp1}
	Let $\lambda$ be a simple eigenvalue of $\mathcal{K}$, and let $\lambda_n^G$ denote the corresponding Galerkin eigenvalue of $P_n \mathcal{K} $. Let $\varphi_n \in \mathcal{R}(E_n)$ be a Galerkin eigenvector, and let $\varphi_n^S$ denote the associated iterated Galerkin eigenvector. 
	Then, for sufficiently large $n$, 
	\begin{eqnarray*} 
		\norm{\varphi_n - E\varphi_n  }_{\infty} = 
		\begin{cases}
			O\left(h^{2r+1} \right) & \text{for } r \geq 1, \\
			O\left(h \right) & \text{for } r = 0,
		\end{cases}
	\end{eqnarray*} 
	
	\begin{eqnarray*}
		\norm{\varphi_n^S - E\varphi_n^S  }_{\infty} =
		\begin{cases}
			O\left(h^{2r+3} \right) & \text{for } r \geq 1, \\
			O\left(h^{2} \right) & \text{for } r = 0,
		\end{cases}
	\end{eqnarray*}
	and 
	\begin{eqnarray*} 
		| \lambda - \lambda_n^G | =
		\begin{cases}
			O\left(h^{2r+3} \right) & \text{for } r \geq 1, \\
			O\left(h^{2} \right) & \text{for } r = 0.
		\end{cases}
	\end{eqnarray*}
\end{theorem}	

In a similar manner, \eqref{Eq: 0021}, \eqref{Eq: 0022}, \eqref{Eq: 0023} and Lemma \ref{lem2}, provide the following convergence rates for the modified and iterated modified Galerkin methods.
\begin{theorem} \label{tc2}
	Let $\lambda$ be a simple eigenvalue of $\mathcal{K}$, and let $\lambda_n^M$ denote the corresponding modified Galerkin eigenvalue of $\mathcal{K}_n^M $. 
	Let $\varphi_n^M \in \mathcal{R}(E_n^M)$ be a modified Galerkin eigenvector, and let $\widetilde{\varphi}_n^M$ denote the associated iterated modified Galerkin eigenvector. 
	Then, for sufficiently large $n$, 
	\begin{eqnarray*} 
		\norm{\varphi_n^M - E\varphi_n^M  }_{\infty} =
		\begin{cases}
			O\left(h^{2r+3} \right) & \text{for } r \geq 1, \\
			O\left(h^{3} \right) & \text{for } r = 0,
		\end{cases}
	\end{eqnarray*}
	
	\begin{eqnarray*} 
		\norm{\widetilde{\varphi}_n^M - E\widetilde{\varphi}_n^M  }_{\infty} = 
		\begin{cases}
			O\left(h^{2r+5} \right) & \text{for } r \geq 1, \\
			O\left(h^{4} \right) & \text{for } r = 0,
		\end{cases}
	\end{eqnarray*}
	and 
	\begin{eqnarray*} 
		| \lambda - \lambda_n^M | = 
		\begin{cases}
			O\left(h^{2r+5} \right) & \text{for } r \geq 1, \\
			O\left(h^{4} \right) & \text{for } r = 0.
		\end{cases}
	\end{eqnarray*}
\end{theorem}	

\subsection{Interpolatory projection based approximations}

First, we establish some important results which will be useful in obtaining the approximate eigenelements using the interpolatory projection. From Rakshit et al. \cite[Proposition 1]{GR-SKS-ASR}, we have the following result:
\begin{proposition} \label{pro2}
	For $r \geq 0$, if $x \in C^{2r+2}[0,1]$, then
	\begin{eqnarray*}
		\norm{\mathcal{K}(I-Q_n)x }_\infty \leq2C_1 \norm{x}_{2r+2, \infty} h^{2r + 2}.
	\end{eqnarray*}
\end{proposition}
\begin{lemma} \label{lem3}
	Let $ \mathcal{K}$ be the linear integral operator defined by \eqref{Eq: 002} with a Green's function-type kernel. Let $Q_n$ be the interpolatory operator onto the approximating space $\mathcal{X}_n$ defined by \eqref{Eq: 0011}. Then for $r \geq 0$,
	\begin{eqnarray*} 
		\norm{(\mathcal{K} - Q_n \mathcal{K}) \mathcal{K} |_{\mathcal{R}(E)} } = O\left(h^{2r +1}\right), \quad 	\norm{ \mathcal{K} (\mathcal{K} - Q_n \mathcal{K}) \mathcal{K} |_{\mathcal{R}(E)} } = O\left(h^{2r +2}\right).
	\end{eqnarray*}
\end{lemma}
\begin{proof}
	The space \(\mathcal{R}(E)\) is invariant by \(\mathcal{K}\), for any \(\varphi \in \mathcal{R}(E)\), we assume \(y= \mathcal{K}\varphi\) and so \(y \in \mathcal{R}(E)\). Consider 
	\begin{equation*}
		\norm{(\mathcal{K} - Q_n \mathcal{K}) \mathcal{K} |_{\mathcal{R}(E)} } = \sup_{\varphi \in \mathcal{R}(E)}\left\{ \norm{(\mathcal{K} - Q_n \mathcal{K}) \mathcal{K} \varphi }_\infty : \norm{\varphi}_\infty =1 \right\}.
	\end{equation*}
	Further, 
	\begin{equation*}
		\norm{(\mathcal{K} - Q_n \mathcal{K}) \mathcal{K} \varphi }_\infty = \norm{(I - Q_n) \mathcal{K} y }_\infty.
	\end{equation*}
	Let \(\varphi \in C^{2r+1}[0,1]\), which implies \(y \in C^{2r+1}[0,1]\) and so \(\mathcal{K}y \in C^{2r+1}[0,1]\). Therefore, using the estimate \eqref{Eq: 0012}, we have
	\begin{equation*}
		\norm{(\mathcal{K} - Q_n \mathcal{K}) \mathcal{K} \varphi }_\infty \leq C_4 \norm{\left(\mathcal{K}y\right)^{(2r+1)}}_\infty h^{2r+1}.
	\end{equation*}
	From inequality \eqref{Eq: 005}, we get
	\begin{align*}
		\norm{\left(\mathcal{K}y\right)^{(2r+1)}}_\infty &\leq C_1 \left(\norm{y}_\infty + \norm{y'}_\infty + \ldots + \norm{y^{(2r-1)}}_\infty\right) \\
		&= C_1 \left(\norm{\mathcal{K}\varphi}_\infty + \norm{(\mathcal{K}\varphi)'}_\infty + \ldots + \norm{(\mathcal{K}\varphi)^{(2r-1)}}_\infty\right).
	\end{align*}
	Following inequalities \eqref{Eq: 004} and \eqref{Eq: 005}, we see that the bound \(	\norm{\left(\mathcal{K}y\right)^{(2r+1)}}_\infty\) does not depend on \(h\). Therefore,
	\begin{equation*}
		\norm{(\mathcal{K} - Q_n \mathcal{K}) \mathcal{K} \varphi }_\infty = O\left(h^{2r+1}\right).
	\end{equation*}
	Hence, 
	\begin{equation*}
		\norm{(\mathcal{K} - Q_n \mathcal{K}) \mathcal{K} |_{\mathcal{R}(E)} } = O\left(h^{2r +1}\right).
	\end{equation*}
	Now, consider the second bound as
	\begin{equation*}
		\norm{ \mathcal{K} (\mathcal{K} - Q_n \mathcal{K}) \mathcal{K} |_{\mathcal{R}(E)} } = \sup_{\varphi \in \mathcal{R}(E)}\left\{ \norm{\mathcal{K}(\mathcal{K} - Q_n \mathcal{K}) \mathcal{K} \varphi }_\infty : \norm{\varphi}_\infty =1 \right\}.
	\end{equation*}
	Assume that \(\varphi \in C^{2r+2}[0,1]\), which implies \(y \in C^{2r+2}[0,1]\) and so \(\mathcal{K}y \in C^{2r+2}[0,1]\). Using Proposition \ref{pro2}, we have
	\begin{equation*}
		\norm{\mathcal{K}(\mathcal{K} - Q_n \mathcal{K}) \mathcal{K} \varphi }_\infty = \norm{\mathcal{K}(I - Q_n ) \mathcal{K} y }_\infty \leq2C_1 \norm{\mathcal{K}y}_{2r+2, \infty} h^{2r+2}.
	\end{equation*}
	Note that 
	\begin{equation*}
		\norm{\mathcal{K}y}_{2r+2, \infty} = \max_{0 \leq i \leq2r +2} \norm{ \left(\mathcal{K}y\right)^{(i)}}_\infty.
	\end{equation*}
	Following inequalities \eqref{Eq: 004} and \eqref{Eq: 005} as discussed above,  we see that the bound \(\norm{ \left(\mathcal{K}y\right)^{(i)}}_\infty \) is independent of \(h\). Therefore, we finally obtain
	\begin{equation*}
		\norm{ \mathcal{K} (\mathcal{K} - Q_n \mathcal{K}) \mathcal{K} |_{\mathcal{R}(E)} } = O\left(h^{2r +2}\right),
	\end{equation*}
	which completes the proof.
\end{proof}
\begin{lemma} \label{lem4}
	Let $ \mathcal{K}$ be the linear integral operator defined by \eqref{Eq: 002} with a Green's function-type kernel. Let $Q_n$ be the interpolatory operator onto the approximating space $\mathcal{X}_n$ defined by \eqref{Eq: 0011}. Then
	\begin{eqnarray*}
		\norm{(\mathcal{K} - \mathcal{K}_n^M) \mathcal{K} |_{\mathcal{R}(E)}} = 
		\begin{cases}
			O\left(h^{2r+2}\right) & \text{for } r \geq 1, \\
			O\left(h^{3}\right) & \text{for } r = 0,
		\end{cases}
	\end{eqnarray*}
	and
	\begin{eqnarray*}
		\norm{\mathcal{K} (\mathcal{K} - \mathcal{K}_n^M) \mathcal{K} |_{\mathcal{R}(E)}} =
		\begin{cases}
			O\left(h^{2r+3}\right) & \text{for } r \geq 1, \\
			O\left(h^{4} \right) & \text{for } r = 0,
		\end{cases}
	\end{eqnarray*}
	where $\mathcal{K}_n^M= Q_n \mathcal{K} + \mathcal{K} Q_n - Q_n \mathcal{K} Q_n.$
\end{lemma}
\begin{proof}
	For any \(\varphi \in \mathcal{R}(E)\),  \( \mathcal{K}\varphi \in \mathcal{R}(E)\). Consider 
	\begin{align*}
		\norm{(\mathcal{K} - \mathcal{K}_n^M) \mathcal{K} |_{\mathcal{R}(E)}} &= \sup_{\varphi \in \mathcal{R}(E)}\left\{ 	\norm{(\mathcal{K} - \mathcal{K}_n^M) \mathcal{K}\varphi}_\infty : \norm{\varphi}_\infty =1 \right\} \\
		&= \sup_{\varphi \in \mathcal{R}(E)}\left\{ \norm{(I - Q_n) \mathcal{K}(I - Q_n)\mathcal{K}\varphi}_\infty : \norm{\varphi}_\infty =1 \right\}.
	\end{align*}
	For \(r \geq 1\), we write
	\begin{equation*}
		\norm{(I - Q_n) \mathcal{K}(I - Q_n)\mathcal{K}\varphi}_\infty \leq\norm{I - Q_n} \norm{\mathcal{K}(I - Q_n)\mathcal{K}\varphi}_\infty.
	\end{equation*}
	Since \(Q_n\) is uniformly bounded, i.e. \(\norm{Q_n} = q < \infty\). For \(\varphi \in C^{2r+2}[0,1]\), it follows that \(\mathcal{K}\varphi \in C^{2r+2}[0,1]\) and using Proposition \ref{pro2}, we have
	\begin{equation*}
		\norm{(I - Q_n) \mathcal{K}(I - Q_n)\mathcal{K}\varphi}_\infty \leq2C_1 (1 + q) \norm{\mathcal{K}\varphi}_{2r+2, \infty} h^{2r+2}.
	\end{equation*}
	Here
	\begin{equation*}
		\norm{\mathcal{K}\varphi}_{2r+2, \infty} = \max_{0 \leq i \leq 2r +2} \norm{ \left(\mathcal{K}\varphi\right)^{(i)}}_\infty.
	\end{equation*}
	Following inequalities \eqref{Eq: 004} and \eqref{Eq: 005}, we observe that the bound \(\norm{ \left(\mathcal{K}\varphi\right)^{(i)}}_\infty \) is independent of \(h\), for \(i=0, 1, \cdots, 2r+2\). Therefore, 
	\begin{equation} \label{Eq: 0035}
		\norm{(I - Q_n) \mathcal{K}(I - Q_n)\mathcal{K}\varphi}_\infty = O\left(h^{2r+2}\right).
	\end{equation}
	When \(r=0\), using \eqref{Eq: 0012}, we have
	\begin{equation*}
		\norm{(I - Q_n) \mathcal{K}(I - Q_n)\mathcal{K}\varphi}_\infty \leq C_4 \norm{\left( \mathcal{K}(I - Q_n)\mathcal{K}\varphi\right)'}_\infty h.
	\end{equation*}
	It follows from Chatelin-Lebbar \cite[Theorem 15]{Cha-Leb2}
	\begin{equation*}
		\norm{\left( \mathcal{K}(I - Q_n)\mathcal{K}\varphi\right)'}_\infty = O\left(h^{2}\right)
	\end{equation*}
	Therefore, from above two estimates, we have
	\begin{equation*} 
		\norm{(I - Q_n) \mathcal{K}(I - Q_n)\mathcal{K}\varphi}_\infty = O\left(h^{3}\right).
	\end{equation*}
	The above bound with \eqref{Eq: 0035} gives the first estimate of the required results. Now, consider
	\begin{equation*}
		\norm{\mathcal{K}(\mathcal{K} - \mathcal{K}_n^M) \mathcal{K} |_{\mathcal{R}(E)}} 
		= \sup_{\varphi \in \mathcal{R}(E)}\left\{  \norm{\mathcal{K}(I - Q_n) \mathcal{K}(I - Q_n)\mathcal{K}\varphi}_\infty : \norm{\varphi}_\infty =1 \right\}.
	\end{equation*}
	For \(r \geq 1\), we write
	\begin{equation*}
		\norm{\mathcal{K}(I - Q_n) \mathcal{K}(I - Q_n)\mathcal{K}\varphi}_\infty \leq \norm{\mathcal{K}(I - Q_n)\mathcal{K}} \norm{(I - Q_n)\mathcal{K}\varphi}_\infty.
	\end{equation*}
	Let $\norm{x}_\infty \leq 1$, then
	\begin{align*}
		\norm{\mathcal{K}(I - Q_n)\mathcal{K}x}_\infty &\leq \norm{\mathcal{K}} \norm{(I - Q_n)\mathcal{K}x}_\infty \\
		&\leq \norm{\mathcal{K}} C_4 \norm{(\mathcal{K}x)''}_{\infty} h^2 \\
		&\leq C_1 C_4 \norm{\mathcal{K}} \norm{x}_{\infty} h^2.
	\end{align*}
	This gives,
	\begin{equation*}
		\norm{\mathcal{K}(I - Q_n)\mathcal{K}} = O\left(h^2\right).
	\end{equation*} 
	Using \eqref{Eq: 0012} and the above bound, we have
	\begin{equation*} 
		\norm{\mathcal{K}(I - Q_n) \mathcal{K}(I - Q_n)\mathcal{K}\varphi}_\infty \leq \norm{\mathcal{K}(I - Q_n)\mathcal{K}} 	\norm{(I - Q_n)\mathcal{K}\varphi}_\infty = 	O\left(h^{2r+3}\right).
	\end{equation*}
	When \(r=0\), by Kulkarni-Nidhin \cite[Proposition 3.7]{RPKTJN}, we obtain
	\begin{equation*} 
		\norm{\mathcal{K}(I - Q_n) \mathcal{K}(I - Q_n)\mathcal{K}\varphi}_\infty = 	O\left(h^{4}\right).
	\end{equation*}
	The above two estimates complete the proof.
\end{proof}
By combining the results \eqref{Eq: 0018}, \eqref{Eq: 0019}, \eqref{Eq: 0020} and Lemma \ref{lem3}, we derive the following orders of convergence for the approximation of eigenvalues and spectral subspaces using collocation and iterated collocation methods.
\begin{theorem} \label{tc1}
	Suppose $\lambda$ is a simple eigenvalue of $\mathcal{K}$, 
	and let $\lambda_n^C$ be the eigenvalue of $Q_n  \mathcal{K} $ converging to $\lambda$.  
	Let $\varphi_n \in \mathcal{R}(E_n)$ be a collocation eigenvector, and let $\varphi_n^S$ denote the corresponding iterated collocation eigenvector. Then, for all $n$ large enough
	\begin{eqnarray*} 
		\norm{\varphi_n - E\varphi_n  }_{\infty} = O\left(h^{2r +1}\right), \quad 	\norm{\varphi_n^S - E\varphi_n^S  }_{\infty} = O\left(h^{2r +2}\right),
	\end{eqnarray*}
	and 
	\begin{equation*} 
		\left| \lambda - \lambda_n^C \right| = O\left(h^{2r +2}\right).
	\end{equation*}
\end{theorem}	

Similarly, \eqref{Eq: 0021}, \eqref{Eq: 0022}, \eqref{Eq: 0023} and Lemma \ref{lem4} give the following orders of convergence for the modified and iterated modified collocation methods.
\begin{theorem} \label{tc2}
	Suppose $\lambda$ is a simple eigenvalue of $\mathcal{K}$, 
	and let $\lambda_n^M$ be the eigenvalue of $\mathcal{K}_n^M$ converging to $\lambda$.  
	Let $\varphi_n^M \in \mathcal{R}(E_n^M)$ be a modified collocation eigenvector, and let $\widetilde{\varphi}_n^M$ denote the corresponding iterated modified collocation eigenvector. Then, for all $n$ large enough
	\begin{eqnarray*} 
		\norm{\varphi_n^M - E\varphi_n^M  }_{\infty} =
		\begin{cases}
			O\left(h^{2r+2} \right) & \text{for } r \geq 1, \\
			O\left(h^{3} \right) & \text{for } r = 0,
		\end{cases}
	\end{eqnarray*}
	
	\begin{eqnarray*} 
		\norm{\widetilde{\varphi}_n^M - E\widetilde{\varphi}_n^M  }_{\infty} = 
		\begin{cases}
			O\left(h^{2r+3} \right) & \text{for } r \geq 1, \\
			O\left( h^{4} \right) & \text{for } r = 0,
		\end{cases}
	\end{eqnarray*}
	and 
	\begin{eqnarray*} 
		| \lambda -  \lambda_n^M | = 
		\begin{cases}
			O\left(h^{2r+3} \right) & \text{for } r \geq 1, \\
			O\left( h^{4} \right) & \text{for } r = 0.
		\end{cases}
	\end{eqnarray*}
\end{theorem}

The complexity, implementation, and efficiency of the linear systems arising from the modified projection method have already been discussed in Kulkarni~\cite{RPK5,RPK6}.  Since \( \mathcal{K}_n^{M} = \pi_n \mathcal{K} + \mathcal{K}\pi_n - \pi_n \mathcal{K}\pi_n,\) and $\pi_n$ maps $\mathcal{X}$ onto the finite-dimensional space $\mathcal{X}_n$, it follows that for any  $x \in \mathcal{X}$, the terms $\pi_n\mathcal{K}x$ and  $\pi_n\mathcal{K}\pi_n x$ belong to $\mathcal{X}_n$, while  $\mathcal{K}\pi_n x \in \mathcal{K}(\mathcal{X}_n)$.  Therefore, $ \mathcal{R}(\mathcal{K}_n^{M}) \subset  Y_n := \mathcal{X}_n + \mathcal{K}(\mathcal{X}_n).$ Since $\dim(\mathcal{X}_n) = n(2r+1)$ and 
$\dim\big(\mathcal{K}(\mathcal{X}_n)\big)  \le \dim(\mathcal{X}_n)$, we obtain $\dim(Y_n) \le 2n(2r+1).$ Therefore, the eigenvalue analysis of $\mathcal{K}_n^{M}$ can be restricted to the finite-dimensional space $Y_n$, and the associated eigenvalue problem reduces to a matrix eigenvalue problem of dimension at most $2n(2r+1)$.  In contrast, the standard projection (or Sloan) method based on $\pi_n \mathcal{K}$ gives a matrix eigenvalue problem of dimension $n(2r+1)$. Compared with this standard projection methods, the modified projection method leads to a quadratic eigenvalue problem, which, after linearization, results in a matrix problem of approximately twice the size of the standard projection methods, thereby increasing the computational cost moderately. However, this additional cost is compensated by the higher-order convergence achieved by the modified projection method, leading to improved accuracy for the same mesh size.


\section{Numerical Results}
To demonstrate the applicability of the theoretical results derived in the preceding section, we now present a numerical example.  We consider the compact integral operator $\mathcal{K}$ defined by
\[
(\mathcal{K}x)(s)=\int_{0}^{1}\kappa(s,t)\,x(t)\,dt,
\]
with the Green's function-type kernel $\kappa$ is given by
\[
\kappa(s,t)=
\begin{cases}
	t(1-s), & 0\le t\le s\le 1,\\[1mm]
	s(1-t), & 0\le s\le t\le 1.
\end{cases}
\]
The kernel $\kappa\in C([0,1]\times[0,1])$, and it is well known that all eigenvalues of $\mathcal{K}$ are simple. The largest eigenvalue is given explicitly by $\lambda=\dfrac{1}{\pi^{2}},$ with the corresponding eigenfunction $\varphi(s)=\sin(\pi s)$.

Let $\mathcal{X}_{n}$ denote the space of piecewise constant functions associated with the uniform partition of $[0,1]$ as
\[
\Delta^{(n)}=\{0=t_{0}<t_{1}<\cdots<t_{n}=1\}, 
\quad t_{j}=\frac{j}{n}, \quad j=0,1,\dots,n.
\]
Throughout the numerical experiments, the experimental order of convergence is computed using the standard formula
\[
\delta=\frac{\log(E_{2n}/E_{n})}{\log 2},
\]
where $E_{n}$ denotes the corresponding numerical error obtained using $n$ subintervals.

\vspace*{0.2cm}
\noindent 
\textbf{Orthogonal projection.} Let $P_n: L^2[0, 1] \mapsto \mathcal{X}_n$ be the orthogonal projection, and let $\lambda_{n}^{G}$ and $\lambda_{n}^{M}$ denote the eigenvalue approximations obtained by the Galerkin and modified Galerkin methods, respectively. The errors in the approximation of the largest eigenvalue, together with the experimentally observed convergence rates, are reported in Table~\ref{table:031}. The numerical results clearly indicate second-order convergence for the standard Galerkin method, while the modified Galerkin method achieves fourth-order convergence, in full agreement with the theoretical results. 

\noindent
The accuracy of the corresponding eigenvector approximations is examined next.  Table~\ref{table:03Gal2} presents the errors measured in the supremum norm,  where the error is computed in terms of the gap between the exact spectral  subspace and its numerical approximation.  The results demonstrate that the Galerkin method exhibits first-order convergence of eigenvectors. The iterated Galerkin and modified Galerkin schemes improve the rate to second and third order, respectively. In particular, the iterated modified Galerkin method achieves fourth-order convergence. These observations clearly illustrate the superconvergence effect produced by iteration and confirm the theoretical 
predictions established in the convergence analysis.

\begin{table}[h!]
	\footnotesize
	\centering
	\setlength{\tabcolsep}{8pt}
	\renewcommand{\arraystretch}{1.2}
	\caption{Error in the approximation of eigenvalues in variants of Galerkin methods}
	\label{table:031}
	\begin{tabular}{|c|c c|c c|}
		\hline
		$n$ & $|\lambda - \lambda_n^G|$ & $\hat\delta_{G}$ & $|\lambda - \lambda_n^M|$ & $\hat\delta_{MG}$  \\[3pt] 
		\hline
		2   & $1.80 \times 10^{-2}$ &            & $3.75 \times 10^{-3}$ &           \\[3pt] 
		4   & $5.04 \times 10^{-3}$ & $1.83$     & $2.95 \times 10^{-4}$ & $3.67$     \\[3pt] 
		8   & $1.29 \times 10^{-3}$ & $1.96$     & $1.96 \times 10^{-5}$ & $3.91$     \\[3pt] 
		16  & $3.25 \times 10^{-4}$ & $1.99$     & $1.25 \times 10^{-6}$ & $3.97$     \\[3pt] 
		32  & $8.13 \times 10^{-5}$ & $2.00$     & $8.06 \times 10^{-8}$ & $3.96$     \\[3pt] 
		64  & $2.03 \times 10^{-5}$ & $2.00$     & $4.91 \times 10^{-9}$ & $4.04$     \\[3pt] 
		\hline
	\end{tabular}
\end{table}


\begin{table}[h!]
	\footnotesize
	\centering
		\setlength{\tabcolsep}{4pt}
	\renewcommand{\arraystretch}{1.2}
	\caption{Error in the approximation of eigenvectors in variants of Galerkin methods}
	\label{table:03Gal2}
	\begin{tabular}{|c|cc|cc|cc|cc|}
		\hline
			$n$ & $\norm{\varphi_n^G - E \varphi_n^G }_\infty$ & $\delta_G$ 
		& $\norm{\varphi_n^S- E \varphi_n^S }_\infty$ & $\delta_{IG}$ 
		& $\norm{\varphi_n^M - E \varphi_n^M }_\infty$ & $\delta_{MG}$ 
		& $\norm{\widetilde{\varphi}_n^{M} - E \widetilde{\varphi}_n^{M}}_\infty$ & $\delta_{IMG}$ \\[3.5pt] 
		\hline
		2   & $6.23\times10^{-1}$ &        & $2.34\times10^{-2}$ &        & $1.11\times10^{-2}$ &        & $3.02\times10^{-4}$ &        \\[3.5pt]
		4   & $3.71\times10^{-1}$ & $0.75$ & $5.81\times10^{-3}$ & $2.01$ & $1.87\times10^{-3}$ & $2.56$ & $1.83\times10^{-5}$ & $4.04$ \\[3.5pt]
		8   & $1.94\times10^{-1}$ & $0.94$ & $1.45\times10^{-3}$ & $2.00$ & $2.50\times10^{-4}$ & $2.90$ & $9.57\times10^{-7}$ & $4.26$ \\[3.5pt]
		16  & $9.78\times10^{-2}$ & $0.98$ & $3.62\times10^{-4}$ & $2.00$ & $3.18\times10^{-5}$ & $2.98$ & $6.03\times10^{-8}$ & $3.99$ \\[3.5pt]
		32  & $4.90\times10^{-2}$ & $1.00$ & $9.06\times10^{-5}$ & $2.00$ & $3.99\times10^{-6}$ & $2.99$ & $3.78\times10^{-9}$ & $4.00$ \\[3.5pt]
		64  & $2.45\times10^{-2}$ & $1.00$ & $2.26\times10^{-5}$ & $2.00$ & $4.99\times10^{-7}$ & $3.00$ & $2.36\times10^{-10}$ & $4.00$ \\[3.5pt]
		\hline
	\end{tabular}
\end{table}

\noindent
\textbf{Interpolatory projection.} Let $ Q_n: C [0, 1] \mapsto \mathcal{X}_n$ denote the interpolatory projection defined by
$
(Q_{n}x)(\tau_{j})=x(\tau_{j}), \quad j=1,2,\dots,n,
$
where the collocation points $\tau_{j}$ are chosen as the midpoints
$
\tau_{j}=\dfrac{2j-1}{2n}, \quad j=1,2,\dots,n.
$
Let $\lambda_{n}^{C}$ and $\lambda_{n}^{M}$ denote the eigenvalue approximations obtained using the collocation and modified collocation methods, respectively. The corresponding eigenvalue errors and experimental convergence rates are reported in Table~\ref{table:032}. The numerical results confirm second-order convergence for the collocation method and fourth-order convergence for the modified collocation method.  Finally, Table~\ref{table:03Col1} reports the errors for eigenvector approximations obtained using collocation-based methods. The results demonstrate that the modified and iterated modified collocation methods significantly improve the convergence rate, as iterated modified collocation eigenvector achieves the highest fourth-order accuracy. Overall, the numerical experiments strongly support the theoretical analysis and highlight the superior performance of the modified projection-based methods.

\begin{table}[h!]
	\footnotesize
	\centering
	\setlength{\tabcolsep}{8pt}
	\renewcommand{\arraystretch}{1.2}
	\caption{Error in the approximation of eigenvalues in variants of collocation methods}
	\label{table:032}
	\begin{tabular}{|c|c c|c c|}
		\hline
		$n$ & $|\lambda - \lambda_n^C|$ & $\hat\delta_{C}$ & $|\lambda - \lambda_n^M|$ & $\hat\delta_{MC}$ \\[3pt] 
		\hline
		4   & $2.44 \times 10^{-3}$ &            & $6.60 \times 10^{-4}$ &           \\[3pt] 
		8   & $6.41 \times 10^{-4}$ & $1.93$     & $4.07 \times 10^{-5}$ & $4.02$     \\[3pt] 
		16  & $1.62 \times 10^{-4}$ & $1.98$     & $2.54 \times 10^{-6}$ & $4.00$     \\[3pt] 
		32  & $4.07 \times 10^{-5}$ & $2.00$     & $1.59 \times 10^{-7}$ & $4.00$     \\[3pt] 
		64  & $1.02 \times 10^{-5}$ & $2.00$     & $9.93 \times 10^{-9}$ & $4.00$     \\[3pt] 
		128 & $2.54 \times 10^{-6}$ & $2.00$     & $6.21 \times 10^{-10}$ & $4.00$     \\[3pt] 
		\hline
	\end{tabular}
\end{table}

\begin{table}[h!]
	\footnotesize
	\centering
		\setlength{\tabcolsep}{4pt}
	\renewcommand{\arraystretch}{1.2}
	\caption{Error in the approximation of eigenvectors in variants of Collocation methods}
	\label{table:03Col1}
	\begin{tabular}{|c|cc|cc|cc|cc|}
		\hline
	$n$ & $\norm{\varphi_n^C - E \varphi_n^C }_\infty$ & $\delta_C$ 
	& $\norm{\varphi_n^S - E \varphi_n^S }_\infty$ & $\delta_{IC}$ 
	& $\norm{\varphi_n^M - E \varphi_n^M }_\infty$ & $\delta_{MC}$ 
	& $\norm{\widetilde{\varphi}_n^M - E \widetilde{\varphi}_n^M }_\infty$ & $\delta_{IMC}$ \\[3.5pt]
		\hline
		4   & $4.14\times10^{-1}$ &    & $7.59\times10^{-2}$ &  & $1.07\times10^{-3}$ &  & $9.31\times10^{-5}$ &  \\[3.5pt]
		8   & $1.99\times10^{-1}$ & 1.06 & $1.91\times10^{-2}$ & 1.99 & $1.37\times10^{-4}$ & $2.96$ & $5.70\times10^{-6}$ & $4.03$ \\[3.5pt]
		16  & $9.85\times10^{-2}$ & 1.01 & $4.81\times10^{-3}$ & 1.99 & $1.73\times10^{-5}$ & $2.99$ & $3.49\times10^{-7}$ & $4.03$ \\[3.5pt]
		32  & $4.91\times10^{-2}$ & 1.00 & $1.18\times10^{-3}$ & 2.02 & $2.16\times10^{-6}$ & $3.00$ & $1.54\times10^{-8}$ & $4.50$ \\[3.5pt]
		64  & $2.45\times10^{-2}$ & 1.00 & $2.97\times10^{-4}$ & 2.00 & $2.70\times10^{-7}$ & $3.00$ & $1.51\times10^{-9}$ & $3.36$ \\[3.5pt]
		128 & $1.23\times10^{-2}$ & 1.00 & $7.41\times10^{-5}$ & 2.00 & $3.37\times10^{-8}$ & $3.00$ & $9.43\times10^{-11}$ & $4.00$ \\[3.5pt]
		\hline
	\end{tabular}
\end{table}

\noindent

For the Green's function-type kernel under consideration, Allouch et al.~\cite{CA-PS-DS-MT} reported that a degenerate kernel method yields eigenvalue approximations  of order $O(h^{3})$. When midpoint collocation with piecewise constant function spaces is used, the corresponding eigenfunction approximation converges with order $O(h)$, which improves to $O(h^{3})$ after applying an iteration technique. In contrast, the modified midpoint collocation method proposed in the present work achieves a higher eigenvalue convergence rate of order $O(h^{4})$. The associated eigenfunction approximation converges with order $O(h^{3})$, and this rate further improves to $O(h^{4})$ under the iterated modified collocation scheme. Thus, for the same discretization setting, the proposed approach provides improved eigenvalue accuracy and enhanced spectral subspace approximation. These numerical observations are in full agreement with the theoretical error estimates established in this manuscript.
	
Moreover, unlike earlier works~\cite{RPK5,RPK6}, which deal with smooth kernels using orthogonal projection and Gauss point based interpolatory projections, the present study focuses on Green’s function-type kernels and adopts orthogonal projection and interpolatory projections based on equidistant collocation points (not necessarily Gauss points) over even degree piecewise polynomial spaces. Extending the superconvergence results of Rakshit et al.~\cite{GR-SKS-ASR} for linear Fredholm integral equations to the associated eigenvalue problems, we show that iteration preserves superconvergence of the eigenfunctions, while the modified projection method yields improved convergence rates for eigenvalue approximation, highlighting the efficiency and flexibility of the proposed framework.


\section{Conclusion}
In this work, we have investigated the convergence behavior of projection-based approximations for eigenvalue problems associated with compact integral operators having Green's kernels. Both classical and modified projection methods were analyzed in a unified framework. We established explicit convergence rates for the approximate eigenvalues and the corresponding spectral subspaces in the supremum-norm.

Our analysis shows that the modified (iterated) projection method produces higher-order convergence of eigenvalues compared to the classical projection method. In particular, the superconvergence phenomenon arising from iteration significantly improves the approximation of spectral subspaces. This enhancement persists even when the underlying Green's kernel lacks full smoothness along the diagonal, thereby extending the scope of earlier results that were primarily restricted to smooth kernels. The theoretical findings are supported by numerical experiments, which confirm that the iterated method achieves higher accuracy with a computational effort. Consequently, the modified projection approach provides an efficient and reliable framework for eigenvalue computations involving Green's kernels. Future research directions include the extension of the present analysis to integral operators with singular or weakly singular kernels, as well as tinvestigation of asymptotic expansions for eigenvalue approximations.


\section*{Declarations}
The authors have no conflict of interest to declare.

\subsection*{Funding Declaration}
The authors declare that no funds, or grants were received during the preparation of this manuscript.

\subsection*{Acknowledgement}
The author Akshay S. Rane acknowledges the UGC Faculty Recharge Program, India.


\end{document}